\documentclass[draft]{amsart}
\usepackage{url}
\usepackage{amsmath,amssymb,amsthm}
\usepackage{enumitem}
\usepackage{fixltx2e}
\usepackage{mathtools}
\usepackage{tikz}
\usepackage{tikz-cd}
\usepackage{xparse}
\usepackage{xspace}
\usepackage{mathtools}
\usepackage{dsfont}
\usepackage[all]{xy}
\usepackage{color}
\usepackage{verbatim}
\usepackage[draft=false]{hyperref}
\usepackage{mathrsfs}
\usepackage{microtype}
\usepackage{bm}

\usepackage{scalerel}
\usepackage{stackengine,wasysym}

\numberwithin{equation}{section}

\let\cal\mathcal

\def\Gscr{{\cal G}}

\def\Lscr{{\cal L}}

\def\Oscr{{\cal O}}

\def\Tscr{{\cal T}}

\def\Yscr{{\cal Y}}

\let\blb\mathbb

\def \AA{{\blb A}}
\def \ZZ{{\blb Z}}

\def \NN{{\blb N}}
\def \RR{{\blb R}}

\def\quot{/\!\!/}

\def\Gr{\operatorname{Gr}}

\def\Lie{\mathop{\text{Lie}}}

\def\Qch{\operatorname{Qch}}

\def\Ext{\operatorname {Ext}}

\def\uEnd{\operatorname {\mathcal{E}\mathit{nd}}}
\def\End{\operatorname {End}}
\def\RHom{\operatorname {RHom}}

\def\End{\operatorname {End}}

\def\r{\rightarrow}

\DeclareMathOperator{\Sym}{Sym}

\theoremstyle{definition}

\newtheorem{lemma}{Lemma}[section]
\newtheorem{proposition}[lemma]{Proposition}
\newtheorem{theorem}[lemma]{Theorem}

\newtheorem{remark}[lemma]{Remark}

\DeclareMathOperator\Hom{Hom}

\def\s{\mathbb{S}}
\makeatletter
\newcommand*\bigcdot{\mathpalette\bigcdot@{.5}}
\newcommand*\bigcdot@[2]{\mathbin{\vcenter{\hbox{\scalebox{#2}{$\m@th#1\bullet$}}}}}
\makeatother

\makeatletter
\newcommand{\pushright}[1]{\ifmeasuring@#1\else\omit\hfill$\displaystyle#1$\fi\ignorespaces}
\newcommand{\pushleft}[1]{\ifmeasuring@#1\else\omit$\displaystyle#1$\hfill\fi\ignorespaces}
\makeatother

\definecolor{ruta2}{rgb}{0.409, 0.459, 0.208}
\usepackage{xcolor}
\hypersetup{
    colorlinks,
    linkcolor={red!50!black},
    citecolor={blue!50!black},
    urlcolor={blue!80!black}
}

\mathchardef\mhyphen="2D

\newcounter{todocounter}
\DeclareDocumentCommand\addreference{g}{\stepcounter{todocounter}\todo[color = blue!30, fancyline]{\thetodocounter. Add reference\IfNoValueF{#1}{: #1}}\xspace}
\DeclareDocumentCommand\checkthis{g}{\stepcounter{todocounter}\todo[color = red!50, fancyline]{\thetodocounter. Check this\IfNoValueF{#1}{: #1}}\xspace}
\DeclareDocumentCommand\fixthis{g}{\stepcounter{todocounter}\todo[color = orange!50, fancyline]{\thetodocounter. Fix this\IfNoValueF{#1}{: #1}}\xspace}
\DeclareDocumentCommand\expand{g}{\stepcounter{todocounter}\todo[color = green!50, fancyline]{\thetodocounter. Expand\IfNoValueF{#1}{: #1}}\xspace}

\makeatletter
\newcommand{\mylabel}[2]{#2\def\@currentlabel{#2}\label{#1}}
\makeatother
\title[]{Tilting bundles on hypertoric varieties} \author[\v{S}pela
\v{S}penko and Michel Van den Bergh]{\v{S}pela \v{S}penko and Michel
  Van den Bergh} \thanks{The first author is a FWO $[$PEGASUS$]^2$
  Marie Sk\l odowska-Curie fellow at the Free University of Brussels
  (funded by the European Union Horizon 2020 research and innovation
  programme under the Marie Sk\l odowska-Curie grant agreement No
  665501 with the Research Foundation Flanders (FWO)). During part of
  this work she was also a postdoc with Sue Sierra at the University
  of Edinburgh.}  \address{Departement Wiskunde, Vrije Universiteit
  Brussel, Pleinlaan $2$, B-1050 Elsene}
\email[]{spela.spenko@vub.be}
\address{Departement WNI, Universiteit Hasselt, Universitaire Campus \\
  B-3590 Diepenbeek} \email[]{michel.vandenbergh@uhasselt.be}
\thanks{The second author is a senior researcher at the Research
  Foundation Flanders (FWO).  While working on this project he was
  supported by the FWO grant G0D8616N: ``Hochschild cohomology and
  deformation theory of triangulated categories''.}

\let\oldmarginpar\marginpar
\def\marginpar#1{\oldmarginpar{\raggedright \tiny #1}}

\begin{document}
\begin{abstract}
Recently McBreen and Webster constructed a tilting bundle
on a smooth hypertoric variety (using reduction to finite characteristic)
and showed that its endomorphism ring is Koszul.

  In this short note we present alternative proofs for these results.
We simply observe that the 
  tilting bundle constructed by Halpern-Leistner and Sam on a generic open GIT substack of the ambient
 linear space
 restricts to a tilting bundle on the hypertoric variety.  
The fact that the hypertoric variety is defined by a quadratic regular sequence then also yields an easy proof of Koszulity.
\end{abstract}

\maketitle
\section{Introduction}
Below $k$ is an algebraically closed ground field of characteristic $0$. 
Let $T$ be torus and let $W$ be a symplectic representation of $T$.
Then $W$ is equipped with a canonical moment map\footnote{If $\langle-,-\rangle$ is  the symplectic bilinear form on $W$, 
then one has $\mu(w)(v)=(1/2)\langle vw,w\rangle$ for $w\in W,v\in \mathfrak{t}$.
}
 $\mu:W\r \mathfrak{t}^\ast$ with $\mathfrak{t}=\Lie(T)$. 
\emph{Throughout we assume that 
 the action of $T$ is faithful} which implies that $\mu$ is surjective and flat.

\medskip

Let $X(T)$ be the characters of $T$. For $\chi\in X(T)$ let
$W^{ss,\chi}$ be the semi-stable part of $W$ with respect to the
linearization $\chi \otimes \Oscr_W$.  Recall that $X(T)_\RR$ is
equipped with a so-called secondary fan such that $\chi$ is in the interior of a
maximal cone if and only if $W^{ss,\chi}/T$ is a (smooth)
Deligne-Mumford
stack \cite[Theorem 14.3.14]{CoxLittleSchenck}.  We will call such $\chi$ \emph{generic}. For generic $\chi$, the DM stack $W^{ss,\chi}/T$ is a crepant resolution of the GIT quotient $W\quot T$ \cite[Lemma A.2 and its proof]{SVdB5}.
Building on the methods developed in \cite{SVdB} Halpern-Leistner and Sam constructed a tilting bundle on $W^{ss,\chi}/T$ \cite{HLSam}.
See Theorem \ref{th:HLS} below.

\medskip

For $\xi\in \mathfrak{t}^\ast$ the \emph{hypertoric variety}
associated to the data $(\chi,\xi)$ is the GIT quotient
$\mu^{-1}(\xi)^{ss,\chi}\quot T$. For $\chi$ generic,
$\mu^{-1}(\xi)^{ss,\chi}/T$ is also a smooth Deligne-Mumford stack
which is a crepant resolution of the hypertoric variety
$\mu^{-1}(\xi)\quot T$ (see \S\ref{sec:proofs}).

The following is our main result.
\begin{theorem} \label{th:mainth}
Let $\chi$ be generic and let $\Tscr$ be the tilting bundle on $W^{ss,\chi}/T$ constructed in 
\cite{HLSam} (see Theorem \ref{th:HLS}).
\begin{enumerate}
\item \label{mainth1}
For $\xi\in \mathfrak{t}^\ast$ the restriction $\Tscr_\xi$ of $\Tscr$ to $\mu^{-1}(\xi)^{ss,\chi}/T$ is a tilting bundle.
\item \label{mainth1.5}
Put $\Lambda=\End_{W^{ss,\chi}/T}(\Tscr)$, $\Lambda_\xi=\End_{\mu^{-1}(\xi)^{ss,\chi}/T}(\Tscr_{\xi})$.  
Then $\Lambda_\xi$ is the quotient of $\Lambda$
by the defining relations of $\mu^{-1}(\xi)$ which form a regular sequence.
\item \label{mainth2}
$\Lambda_\xi$ is a ``non-commutative crepant resolution'' \cite{VdB04} of 
$k[\mu^{-1}(\xi)\quot T]$. 
\item \label{mainth3}
$\Lambda_0$ is Koszul when graded using
the dilation $G_m$-action on $W$.
\end{enumerate}
\end{theorem}
The fact that $\mu^{-1}(\xi)^{ss,\chi}/T$ admits a tilting bundle with
Koszul endomorphism ring when $\xi=0$ was proved in
\cite{McBreenWebster} using the Bezrukavnikov-Kaledin method based on
reduction mod $p$.
\begin{remark} 
$\Lambda$ is an NCCR for $W\quot T$ (see Remark \ref{rem:NCCR}) but it is not Koszul, except in trivial cases. This follows from Proposition \ref{prop:koszul} 
applied in the same way as in the proof of Theorem \ref{th:mainth}\eqref{mainth3}.
For example in the case of the conifold ($G_m$ acting with weights $1,1,-1,-1$ on a 4-dimensional representation) $\Lambda$
has cubic relations. This is in fact expected as explained in the next remark.
\end{remark}
\begin{remark}
From the fact that $\Lambda_0$ is Koszul one obtains in particular that it is quadratic. 
Using the explicit form of $\Tscr$ (see Theorem \ref{th:HLS})
one  then quite easily obtains a quiver presentation of $\Lambda_0$. See \cite[Corollary 3.18]{McBreenWebster}. This presentation
of $\Lambda_0$ may be lifted to a quadratic non-homogeneous presentation of $\Lambda$ as $\Sym(\mathfrak{t})$-algebra (with $|\mathfrak{t}|=2$). This then yields a more complicated presentation
of $\Lambda$ as $k$-algebra, possibly involving extra quadratic generators and cubic and even quartic relations.
\end{remark}
Similar results as those of McBreen and Webster have been announced by Tatsuyuki Hikita \cite{Hikita}.
\section{Acknowledgement}
This note is mostly the result of a discussion with Travis Schedler 
during the conference  
``Quantum geometric and algebraic representation theory''
at the Hausdorff Institute in October 2017, which was attended by both authors.  We are grateful to Travis for his interest in this work. We would also like to thank Theo Raedschelders for drawing our attention to the work of McBreen and Webster.
\section{The tilting bundle on $W^{ss,\chi}/T$}
\label{sec:tilting}
Recall that a tilting bundle on an algebraic stack $\Yscr$ is a vector bundle $\Tscr$ such that 
$\Ext^{>0}_\Yscr(\Tscr,\Tscr)=0$ and such that $\Tscr$  generates $D_{\Qch}(\Yscr)$ in the sense that $\Tscr^\perp=0$.

\medskip

For benefit of the reader we now describe explicitly the tilting bundle on $W^{ss,\chi}/T$ constructed in \cite{HLSam}. The construction in fact only requires
that $W$ is \emph{quasi-symmetric} \cite{SVdB}, i.e. the sum of weights of $W$ on each line
through the origin is zero.

Let $(\beta_i)_{i=1}^{2e}$ be the $T$-weights of $W$.
Let $\varepsilon\in X(T)$. Put
\[
\Sigma=\left\{\sum_i a_i\beta_i\mid a_i\in ]-1,0]\right\}\subset X(T)_\RR, \quad\quad \bar{\Sigma}_\varepsilon=\bigcup_{r>0}\bar{\Sigma}\cap (r\varepsilon +\bar{\Sigma}).
\]
The following result indicates the combinatorial significance of the zonotope $\bar{\Sigma}$.
\begin{proposition}[{\cite[Proposition 2.1]{HLSam}}]
Assume that $W$ is quasi-symmetric. 
A character $\chi\in X(T)$ is  generic if and only if it is not parallel to any face of $\bar{\Sigma}$.
\end{proposition}
For $\varepsilon \in X(T)$ put
$
\Lscr_{\varepsilon}=X(T)\cap (1/2)\bar{\Sigma}_\epsilon
$.
The following is one of the main results of \cite{HLSam}.
\begin{theorem}[{\cite[Corollary 4.2]{HLSam}.}]
\label{th:HLS}
Assume that $W$ is quasi-symmetric and that
$\chi\in X(T)$ is generic. Then for any generic $\varepsilon\in X(T)$
\[
\Tscr=\bigoplus_{\mu\in\Lscr_\varepsilon}\Oscr_{W^{ss,\chi}}\otimes_k \mu
\]
defines a tilting 
 bundle on $W^{ss,\chi}/T$.
\end{theorem}
\begin{remark}\label{rem:NCCR} If $W$ is ``generic''
in the
  sense of \cite[Definition 1.3.4]{SVdB}, i.e. the complement of
  $\{x\in W\mid \text{$Tx$ is closed and
    $\operatorname{Stab} x=\{e\}$}\}$ has codimension $\ge
  2$
 then
  $\Lambda=\End_{W^{ss,\chi}/T}(\Tscr)=(\End_k(\bigoplus_{\mu \in
    \Lscr_\varepsilon} \mu)\otimes_k k[W])^T$
  is the NCCR of $W\quot T$ constructed in \cite{SVdB}.
  Assuming that $T$ acts faithfully and that $W$ is
  symplectic, one can always reduce to the case of generic $W$.

If $W$ is not generic then all except for two $T$-weights of $W$ (whose sum is $0$) lie in a hyperplane. Let $W'$ be a vector subspace of $W$ spanned by the weight vectors whose weights lie in this hyperplane. Then $T=T'\times G_m$, where $G_m$ acts trivially on $W'$ and the action of $T'$ is faithful on $W'$. This induces the decomposition of $X(T)$ and $\mathfrak{t}$. 
The corresponding projections onto $X(T')$ and $\mathfrak{t'}$ will be denoted by  $'$. 
We have $W\quot T=W'\quot T'\times_k \AA^1$, and for a generic $\chi$ also $W^{ss,\chi}/T=W^{ss,\chi'}/T'\times_k\AA^1$. Moreover, $\Lscr_{\varepsilon}=\Lscr'_{\varepsilon'}\times \{m\}$ ($\Lscr'_{\varepsilon'}$ with respect to $(W',T')$) for some $m\in \ZZ$. Repeating if necessary we reduce to the case that $W$ is generic.

We also note that the corresponding hypertoric varieties do not change. Let $\mu,\mu'$ correspond to the moment maps associated to $W$, $W'$, resp. Then $\mu^{-1}(\xi)\quot T=\allowbreak\mu'^{-1}(\xi')\quot T'$, $\mu^{-1}(\xi)^{ss,\chi}/ T=\allowbreak\mu'^{-1}(\xi')^{ss,\chi'}/ T'$ (for a generic $\chi$). 
\end{remark}

\section{Proofs} 
\label{sec:proofs}
We revert to the setting of the introduction, i.e.\ $W$ is a  
symplectic
representation of the torus $T$.
Assume $\chi\in X(T)$ is a generic character. As explained in the introduction
$W^{ss,\chi}/T$ is a DM stack which yields a (stacky) crepant 
resolution $\pi:W^{ss,\chi}/T\r W\quot T$ of the Gorenstein variety
\cite[Cororllary 13.3]{St5} $W\quot T$. The fact that the stabilizers of the
$T$-action on $W^{ss,\chi}$ are finite
yields by the defining property of the
moment map that $\mu{\mid} W^{ss,\chi}$ is smooth. Hence for every
$\xi\in \mathfrak{t}^\ast$, $\mu^{-1}(\xi)^{ss,\chi}/T$ is also a
smooth Deligne-Mumford stack.

Since~$T$ is commutative, $\mu^{-1}(\xi)\subset W$ is cut out by a
\emph{$T$-invariant} regular sequence of global sections (see
\cite[Theorem 2]{Vinberg} and \cite[Proposition 9.4]{Schwarz1995}). In
particular $\mu^{-1}(\xi)\quot T$ is cut out by a regular sequence in
$W\quot T$ and the same regular sequence cuts out
$\mu^{-1}(\xi)^{ss,\chi}/T$ in $W^{ss,\chi}/T$.  It follows
from this that $\mu^{-1}(\xi)\quot T$ is Gorenstein and that $\pi$
restricts to a crepant resolution 
$\pi_\xi:\mu^{-1}(\xi)^{ss,\chi}/T\to\allowbreak \mu^{-1}(\xi)\quot
T$.

Finally note that $\mu^{-1}(\xi)\quot T$ is normal. For $\xi=0$ this is shown
in the proof of \cite[Proposition 4.11]{MR3001787}. For
the benefit of the reader we give an elementary proof valid for general $\xi$.
Using Remark
\ref{rem:NCCR} we first reduce to the generic case and in that case we will show that
$\mu^{-1}(\xi)$ is normal (the quotient of a normal variety is normal).  Since
$\mu^{-1}(\xi)$ is Cohen-Macaulay, it suffices by Serre's normality
criterion to prove that $\mu^{-1}(\xi)$ is regular in codimension $1$.
Thus, it is sufficient to prove that codimension in $\mu^{-1}(\xi)$ of the intersection of $\mu^{-1}(\xi)$
with the
non-smooth locus of $\mu$
is $\geq 2$.  We will in fact verify that this codimension is $\ge 3$!

By the defining property of the moment map,
the non-smooth locus of $\mu$ is the locus of points in $W$ whose stabilizer has dimension
$>0$.  This is a union of $T$-invariant subspaces of $W$. Each such
subspace $W'$ is again a symplectic vector space with a faithful
action of a quotient torus $T/T'$ of $T$, where~$T'$ denotes the generic
stabilizer of~$W'$ (hence by hypotheses $\dim T'>0$) and moreover $\mu$ restricts to a moment map for the $T/T'$-action. 
Thus, as above, $W'\cap \mu^{-1}(\xi) $ is cut out by 
a regular sequence of length $\dim T/T'$ in $W'$.
Hence the dimension of $W'\cap \mu^{-1}(\xi)$ is equal to
$c:=\dim W'-\dim T/T'$.

Choose
a filtration $W=W_0\supsetneq W_1\supsetneq\cdots
\supsetneq W_t=W'$ of $T$-invariant symplectic subspaces
such that $\dim W_{i+1}=\dim W_i-2$ and let $T_{i}$ be the
 generic stabilizer of~$W_i$. Write $c_i=\dim W_i-\dim T/T_i=\dim W_i\cap \mu^{-1}(\xi)$.

The dimension of  $T/T_i$ is equal
to the rank of the sublattice of $X(T)$ spanned
by the weights of $W_i$. Hence by genericity we have $\dim T/T_1=\dim T$, and thus $c_1=c_0-2$. Since  $\dim T'>0$
we have  $t\ge 2$.
In addition for $i\ge 1$ we have $c_{i+1}\le c_i-1$ ($\dim W_i$ goes down by $2$ but $\dim T/T_i$ goes at most down by $1$). So we get $c=c_t\le c_0-3$
and we are done.

\begin{proof}[Proofs of Theorem \ref{th:mainth}\eqref{mainth1} and  \ref{th:mainth}\eqref{mainth1.5}]
Let $i:\mu^{-1}(\xi)^{ss,\chi}/T\r W^{ss,\chi}/T$ be the inclusion. We have to prove that $\Tscr_\xi=i^\ast \Tscr$ is a tilting bundle on $\mu^{-1}(\xi)^{ss,\chi}/T$. We have
\[
\Ext^\ast_{\mu^{-1}(\xi)^{ss,\chi}/T}(i^*\Tscr,i^*\Tscr)=
\Ext^\ast_{W^{ss,\chi}/T}(\Tscr,i_\ast i^*\Tscr).
\]
Now $\mu^{-1}(\xi)^{ss,\chi}$ is cut out in $W^{ss,\chi}$ by an invariant regular sequence. Tensoring the corresponding Koszul resolution of $i_\ast \Oscr_{\mu^{-1}(\xi)^{ss,\chi}}$ with $\Tscr$ we obtain a left resolution $K_\bullet$ of $i_\ast i^*\Tscr$ which consists
of direct sums of $\Tscr$. Using the fact that $\Ext^{>0}_{W^{ss,\chi}/T}(\Tscr,\Tscr)=0$ we obtain 
\begin{equation}
\label{eq:kozres}
\Ext^\ast_{W^{ss,\chi}/T}(\Tscr,i_\ast i^*\Tscr)=H^\ast(\Hom_{W^{ss,\chi}/T}(\Tscr,K_{\bullet})).
\end{equation}
Since $K_\bullet$ lives in degree $\le 0$ this implies $\Ext^{>0}_{\mu^{-1}(\xi)^{ss,\chi}/T}(i^*\Tscr,i^*\Tscr)=0$. Since also tautologically  $\Ext^{<0}_{\mu^{-1}(\xi)^{ss,\chi}/T}(i^*\Tscr,i^*\Tscr)=0$
we obtain that the right-hand side of \eqref{eq:kozres} is in fact a resolution of $\Lambda_\xi$, proving Theorem \ref{th:mainth}\eqref{mainth1.5}.

To finish the proof of Theorem \ref{th:mainth}\eqref{mainth1} we have to check the 
 generation property. Assume that $\RHom_{\mu^{-1}(\xi)^{ss,\chi}/T}(i^*\Tscr,\Gscr)\allowbreak =0$ for $\Gscr\in D_{\Qch}(\mu^{-1}(\xi)^{ss,\chi}/T)$. Then by adjunction
$i_\ast \Gscr\in \Tscr^\perp$ and hence $i_\ast\Gscr=0$. It follows that $\Gscr=0$.
\end{proof}
\begin{proof}[Proof of Theorem \ref{th:mainth}\eqref{mainth2}]
  Since $\pi_\xi$ is a crepant resolution, it follows that
  $\Lambda_\xi=\allowbreak\pi_*\uEnd_{\mu^{-1}(\xi)^{ss,\chi}/T}(\Tscr_\xi)$ is an
  NCCR of $\mu^{-1}(\xi)\quot T$ (see e.g. \cite[Corollary 4.3]{SVdB4}
  and Corollary 4.7 in loc.\ cit.\ with its proof, together with  Remark \ref{rem:NCCR}).
\end{proof}
The proof of  Theorem \ref{th:mainth}\eqref{mainth3} will be based on the following
more general criterion.
\begin{proposition}
\label{prop:koszul} Let $\Delta$ be a $\NN$-graded homologically homogeneous $k$-algebra
  \cite{BH} (see also \cite[\S3]{VdB04}, \cite[\S2]{stafford2008}) such that $\Delta_0$ is semi-simple. Assume that $\Delta$ is
finite as a module over a central subring $R$ of Krull-dimension $d$. Then
the following are equivalent.
\begin{enumerate}
\item $\Delta$ is Koszul.
\item The invertible $\Delta$-bimodule $\omega_\Delta:=\Hom_R(\Delta,\omega_R)$ (see \cite[Proposition 2.6]{stafford2008}) is generated as right module in degree
$d$.
\end{enumerate}
Moreover in (2) we may replace ``right'' by ``left''.
\end{proposition}
\begin{proof} 
  Note that by \cite[Lemma 2.4(2), Proposition 2.6]{stafford2008}
  $\omega_\Delta[d]$ is a ``rigid dualizing complex'' for
  $\Delta$. Hence by the uniqueness of rigid dualizing complexes
  together with the proof of\footnote{The reference \cite{van1997existence} is concerned
with the case $\Delta_0=k$ but this hypothesis is not used in an essential way.} \cite[Proposition
  8.2(2)]{van1997existence} we obtain as $\Delta$-bimodules:
  $\omega_{\Delta}[d]=R\Gamma_{\Delta_{\ge 1}}(\Delta)^\vee$ where $(-)^\vee$
  denotes the graded $k$-dual and $R\Gamma_{\Delta_{\ge 1}}$ is the derived functor of $\varinjlim_n \RHom_\Delta(\Delta/\Delta_{\ge n},-)$. Moreover by \cite[Theorem 5.1]{van1997existence} for $M\in \allowbreak D(\Gr(\Delta))$ we 
obtain as objects in $D(\Gr(\Delta^\circ))$
\[
\RHom_\Delta(M,\omega_\Delta[d])\cong R\Gamma_{\Delta_{\ge 1}}(M)^\vee\,.
\]
Applying this with a graded simple $\Delta$-module $\s$, concentrated in degree zero we find
\[
\RHom_\Delta(\s,\omega_\Delta[d])=\s^*;
\]
i.e. $\RHom_\Delta(-,\omega_\Delta[d])$ restricts to the canonical bijection between the graded simple left and right $\Delta$-modules,
concentrated in degree zero.

We will now prove $(2)\Rightarrow(1)$.
Since $\omega_\Delta$ is invertible and generated in degree
$d$ we have that $\omega_\Delta\otimes_\Delta \s=\s^{\dagger}(-d)$ where
$\s^\dagger$ is a graded simple $\Delta$-module concentrated in degree zero. Hence we obtain
\begin{equation}
\label{eq:rhomsimple}
\RHom_\Delta(\s,\Delta)=\RHom_\Delta(\omega_\Delta[d]\otimes_\Delta \s,\omega_\Delta[d])= \s^{\dagger *}(d)[-d]
\end{equation}
Let  $P_\bullet=\cdots\to P_{1}\to P_0$ be the minimal graded projective resolution of $\s$ over $\Delta$. As $\Delta$ is homologically homogeneous, 
 $P_\bullet$ is of length $d$, and by \eqref{eq:rhomsimple}
$\Hom_\Delta(P_\bullet,\Delta)[d]$ 
is a projective resolution of $\s^{\dagger *}(d)$. Hence
 $P_{d}=P(\s^\dagger)(-d)$, where $P(\s^\dagger)$ is the graded projective cover of $\s^\dagger$.

 Let $f_i$ be the minimal degree of an element in $P_{i}$.  We claim
 $f_{i-1}<f_i$. To see this note that
the graded Jacobson radical
 of $\Delta$ is equal to $\Delta_{\ge 1}$ (as $\Delta_0$ is semi-simple). Hence 
the image of $d_i:P_i\r P_{i-1}$ is in $\Delta_{\ge 1}P_{i-1}$.

An element $x$ of degree $f_i$
 in $P_{i}$ is not in $\Delta_{\ge 1}P_{i}$ and therefore it is not in the image of $d_{i+1}$. As
$P_\bullet$ is acyclic it follows that $d_i(x)\neq 0$. Writing $d_i(x)=\sum_{j}l_jy_j$ with 
homogeneous
$l_j\in \Delta_{\ge 1}$, $y_j\in P_{i-1}$, we see that some $y_j$ must be non-zero. The fact
$l_j$ has strictly positive degree implies
 $\deg(y_j)<\deg(x)$ which proves our claim.

From the inequalities
\[
d=f_d>f_{d-1}>\cdots>f_1>f_0=0
\]
we obtain $f_i=i$.  

We now use a dual argument. Let $f^\ast_i$ be the maximal degree of a generator of~$P_{i}$. 
Using the dual resolution $\Hom_\Delta(P_\bullet,\Delta)(-d)[d]$ of $\s^\dagger$ (which is also minimal)
we get that $-d+f^\ast_{d-i}=-i$, i.e. $f^\ast_i=i$. Therefore $P_{i}$ is purely generated in degree $i$
and so the resolution $P_\bullet$ is linear.

\medskip

The implication $(1)\Rightarrow (2)$ is similar. From the nature of the minimal resolutions of the
simples we get $\omega_\Delta\otimes_\Delta \s=\s^\dagger(d)$ for all $\s$. This implies that $\omega_\Delta$
is right generated in degree $d$.

\medskip

 Since Koszulity is left right symmetric, the hypothesis of right generation in (2) may indeed be replaced by left
generation.
\end{proof}
\begin{proof}[Proof of Theorem \ref{th:mainth}\eqref{mainth3}]
  Let $s=\dim T$, $2e=\dim W$ and $d=\dim \mu^{-1}(0)\quot T$. Thus
  $d=2e-2s$.  

By Theorem \ref{th:HLS} have $\Lambda_0=(\End_k(\bigoplus_{\mu\in \Lscr_\xi} \mu)\otimes_k k[\mu^{-1}(0)])^T$.
Hence $\Lambda_0$
is $\NN$-graded and its part of degree zero is semi-simple.\footnote{This is the only place where
we use the specific form of $\Tscr$.}
Therefore by Proposition \ref{prop:koszul} we have to prove that $\omega_{\Lambda_0}$ is generated in degree $d$.

Put $R=k[W\quot T]$, $R_0=k[\mu^{-1}(0)\quot T]$.
We have
\[
\omega_{R_0}=\Ext^{s}_R(R_0,\omega_R)=\Ext^{s}_R(R_0,R(-2e))=R_0(2s-2e)=R_0(-d)
\]
where the first equality is the adjunction formula,
the second equality follows from \cite[Corollary 13.3]{St5} and the third equality follows
from the fact that $R_0$ is cut out from $R$ by a quadratic regular sequence of length $s$.
As $\Lambda_0$ is an NCCR by Theorem \ref{th:mainth}\eqref{mainth2} we have $\Hom_{R_0}(\Lambda_0,R_0)=\Lambda_0$.
It follows that $\omega_{\Lambda_0}=\Hom_{R_0}(\Lambda_0,\omega_{R_0})=\Lambda_0(-d)$, finishing the proof.
\end{proof}
\providecommand{\bysame}{\leavevmode\hbox to3em{\hrulefill}\thinspace}
\providecommand{\MR}{\relax\ifhmode\unskip\space\fi MR }
\providecommand{\MRhref}[2]{%
  \href{http://www.ams.org/mathscinet-getitem?mr=#1}{#2}
}
\providecommand{\href}[2]{#2}

\end{document}